\documentclass{mfat}
\pagespan{256}{265}

\usepackage{mmap}
\usepackage[utf8]{inputenc}

\newtheorem{theorem}{Theorem}[section]

\newtheorem{lemma}[theorem]{Lemma}
\newtheorem{corollary}[theorem]{Corollary}

\newtheorem*{thA}{Theorem A}
\newtheorem*{claim}{Claim}
\newtheorem*{example}{Example}

\begin{document}

\title[Some results on order bounded awDP operators]
      {Some results on order bounded almost weak Dunford-Pettis operators}

\author{Nabil Machrafi}
\address
{Universit\'{e} Ibn Tofa\"{\i}l, Facult\'{e} des Sciences,
D\'{e}partement de Math\'{e}\-matiques, B.P. 133, K\'{e}nitra
14000, Maroc.} \email{nmachrafi@gmail.com}

\author{Aziz Elbour}
\address
{Universit\'{e} Moulay Isma\"{\i}l, Facult\'{e} des Sciences et
Techniques, D\'{e}partement de Math\'{e}\-matiques, B.P. 509,
Errachidia 52000, Maroc.} \email{azizelbour@hotmail.com}

\author{Mohammed Moussa}
\address
{Universit\'{e} Ibn Tofa\"{\i}l, Facult\'{e} des Sciences,
D\'{e}partement de Math\'{e}\-matiques, B.P. 133, K\'{e}nitra
14000, Maroc.} \email{mohammed.moussa09@gmail.com}

\subjclass[2010]{46A40, 46B40, 46B42}
\date{29/09/2015; \ \  Revised 25/02/2016}
\keywords{Almost weak Dunford-Pettis operator, weak Dunford-Pettis
operator, almost Dunford-Pettis set, almost (L)-set, Banach
lattice.}

\begin{abstract} We give some new characterizations of almost
weak Dunford-Pettis ope\-rators and we investigate their
relationship with weak Dunford-Pettis operators.
\end{abstract}

\maketitle

\section{Introduction and notations}

Throughout this paper, $X,$ $Y$ will denote real Banach spaces, and $E,\,F$
will denote real Banach lattices. $B_{X}$ is the closed unit ball of $X$. We
mean by operator between Banach spaces, a bounded linear application.

A real vector space $E$ is said to be a \emph{partially} \emph{ordered
vector space} whenever it is equipped with a partial order relation $\geq $
(i.e., a reflexive, antisymmetric, and transitive binary relation on $E$)
that is compatible with the algebraic structure of $E$ in the sense that it
satisfies the following two axioms:

\begin{enumerate}
\item If $x\geq y$, then $x+z\geq y+z$ holds for all $z\in E$.

\item If $x\geq y$, then $\lambda x\geq \lambda y$ holds for all $\lambda
\geq 0$.
\end{enumerate}

An alternative notation for $x\geq y$ is $y\leq x$. The \emph{positive cone}
of $E$, denoted by $E^{+}$, is the set of all positive vectors of $E$, i.e.,
$E^{+}:=\left\{ x\in E:x\geq 0\right\} $. If furthermore, every set $\left\{
x,y\right\} \subset E$ has a supremum $\sup \left\{ x,y\right\} =x\vee y$
(or equivalently it has an infimum $\inf \left\{ x,y\right\} =x\wedge y$)
then $E$ is called a \emph{vector lattice} (or \emph{Riesz space}). The
elements%
\begin{equation*}
x^{+}=x\vee 0\text{, \ \ }x^{-}=\left( -x\right) \vee 0\text{, \ \ }%
\left\vert x\right\vert =x\vee \left( -x\right)
\end{equation*}
are called the \emph{positive part}, \emph{negative part}, and \emph{modulus}
of the element $x$, respectively. Note that $x=x^{+}-x^{-}$ and $%
|x|=x^{+}+x^{-}$. If $x,y\in E$ and $x\leq y$, then the \emph{order interval}
$\left[ x,y\right] $ is defined by%
\begin{equation*}
\left[ x,y\right] =\left\{ z\in E:x\leq z\leq y\right\} \text{.}
\end{equation*}

A subset $A\subset E$ is said to be \emph{order bounded} if it is
contained
in some order interval. It is said to be \emph{solid} if conditions $%
\left\vert x\right\vert \leq \left\vert a\right\vert $ and $a\in A$ imply $%
x\in A$. The smallest solid set containing a set $A$ is called the \emph{%
solid hull} of $A$ and denoted by $sol\left( A\right) $. We have%
\begin{equation*}
sol\left( A\right) =\left\{ x\in E:\left\vert x\right\vert \leq \left\vert
a\right\vert \text{ for some }a\in A\right\} .
\end{equation*}

The elements $x,y\in E$ are called \emph{disjoint} if $\left\vert
x\right\vert \wedge \left\vert y\right\vert =0$. A generalized sequence $%
\left( x_{\alpha }\right) $ in a vector lattice $E$ (i.e. a function $\alpha
\rightarrow x_{\alpha }$ from some an upward directed set $I$ to $E$) is
called disjoint, if $\left\vert x_{\alpha }\right\vert \wedge \left\vert
x_{\beta }\right\vert =0$, $\alpha \neq \beta $. We will use the notation $%
x_{\alpha }\perp x_{\beta }$ to mean that the generalized sequence $\left(
x_{\alpha }\right) $ is disjoint. A collection $\left( e_{i}\right) \subset
E^{+}\backslash \left\{ 0\right\} $ is called a \emph{complete disjoint
system} if%
\begin{equation*}
e_{i}\wedge e_{j}=0,\quad i\neq j\quad \text{and} \quad
e_{i}\wedge \left\vert x\right\vert =0\quad \text{for all}\ i \
\text{implies}\quad  x=0\text{.}
\end{equation*}

A positive non-zero element $x$ of an \emph{Archimedean vector lattice} $E$
(i.e., such vector lattices that $\inf_{n}\left\{ \frac{1}{n}x\right\} =0$
for every $x\in E^{+}$) is called \emph{discrete}, if%
\begin{equation*}
u,v\in \left[ 0,x\right] \quad \text{and} \quad u\wedge v=0 \quad \text{imply} \quad u=0
 \quad \text{or} \quad
v=0\text{.}
\end{equation*}

An Archimedean vector lattice $E$ is called \emph{discrete} (or \emph{atomic}%
), if $E$ has a complete disjoint system consisting of discrete elements, or
equivalently, every non-trivial interval $\left[ 0,x\right] $ contains a
discrete element. A norm $\left\Vert .\right\Vert $ on a vector lattice $%
\left( E,\leq \right) $ is said to be a \emph{lattice norm} whenever%
\begin{equation*}
|x|\leq |y| \quad \text{implies} \quad \left\Vert x\right\Vert
\leq \left\Vert y\right\Vert .
\end{equation*}

A vector lattice equipped with a lattice norm is known as a \emph{normed
vector lattice}. If a normed vector lattice is also norm complete, then it
is referred to as a \emph{Banach lattice}. Note that If $E$ is a Banach
lattice, its topological dual $E^{\prime }$, endowed with the dual norm, is
also a Banach lattice. A norm $\left\Vert .\right\Vert $ of a Banach lattice
$E$ is called \emph{order continuous} if for each generalized sequence $%
(x_{\alpha })\subset E$, $x_{\alpha }\downarrow 0$ implies $\left\Vert
x_{\alpha }\right\Vert \rightarrow 0$, where the notation $x_{\alpha
}\downarrow 0$ means that the generalized sequence $(x_{\alpha })$ is
decreasing and $\inf \left\{ x_{\alpha }\right\} =0$. A Banach lattice $E$
is said to be a \emph{Kantorovich--Banach space} (briefly \emph{KB-space})
whenever every increasing norm bounded sequence of $E^{+}$ is norm
convergent. Note that every KB-space has order continuous norm. The lattice
operations in a Banach lattice $E$ are said to be \emph{sequentially weakly}%
\ \emph{continuous}, if for every weakly null sequence $(x_{n})$ in $E$ we
have $\left\vert x_{n}\right\vert \rightarrow 0$ for $\sigma (E,E^{\prime })$%
. For a sequence $\left( x_{n}\right) \subset E$ of a Banach lattice, the
following fact will be used throughout this paper (see \cite[Theorem 4.34]%
{AB3}):%
\begin{equation*}
x_{n}\overset{\sigma (E,E^{\prime })}{\rightarrow }0, \quad
x_{n}\perp x_{m} \quad \text{implies} \quad \left\vert
x_{n}\right\vert \overset{\sigma (E,E^{\prime })}{\rightarrow
}0\text{.}
\end{equation*}

Recall that a subset $A$ of a Banach space $X$ is called a \emph{%
Dunford-Pettis (DP) set}, if each weakly null sequence $\left( f_{n}\right) $
in $X^{\prime }$ converges uniformly to zero on $A$. In his paper \cite{L},
T.Leavelle considered the dual version of DP sets, so-called \emph{(L) sets}%
, that is, each subset $B$\ of the topological dual $X^{\prime }$,\ on which
every weakly null sequence $(x_{n})$ in\ $X$\ converges uniformly to zero.
Recently, the authors of \cite{KB} and\ \cite{aqz}\ considered respectively,
the disjoint versions of DP sets and (L)-sets, which\textbf{\ }are
respectively called \emph{almost Dunford-Pettis} (almost DP) sets and \emph{%
almost (L)-sets}. From \cite{KB} (resp.~\cite{aqz}), a norm bounded subset $%
A\subset E$ (resp. $B\subset E^{\prime }$) is said to be almost DP
(resp. almost (L)-) set, if every disjoint weakly null sequence
$(f_{n})\subset E^{\prime }$ (resp. $(x_{n})\subset E$ ) converges
uniformly to zero on $A$ (resp. $B$). Clearly, every DP set in a
Banach lattice (resp. (L)-set in the topological dual of a Banach
lattice) is an almost DP (resp. almost (L)-) set . But the
converse is false in general for the two latter classes of sets.

A Banach lattice $E$ is said to have the (\emph{positive}) \emph{Schur}
\emph{property} if every (positive) weakly null sequence $\left(
x_{n}\right) \subset E$ is norm null. Furthermore, A Banach space $X$ is
said to have the \emph{Dunford-Pettis} \emph{(DP)} \emph{property} if each
relatively weakly compact set in $X$ is a DP set, alternatively, $%
f_{n}\left( x_{n}\right) \rightarrow 0$ for every weakly null sequences $%
(x_{n})\subset X$ , $(f_{n})\subset X^{\prime }$.

An operator $T:E\rightarrow F$ between two vector lattices is called an
\emph{order bounded operator}, if it maps order bounded subsets of $E$ into
an order bounded ones of $F$. It is \emph{positive} if $T(E^{+})\subset
F^{+} $. The positive operators between two vector lattices generate the
vector space of all \emph{regular operators}, i.e., operators that are
written as a difference of two positive operators. Note that a regular
operator between two vector lattices need not be order bounded (see example
of H. P. Lotz \cite[Example 1.16]{AB3}). An operator $T:X\rightarrow Y$
between two Banach spaces is said to be \emph{Dunford-Pettis}, if $T$
carries each weakly null sequence $(x_{n})$ in $X$ to a norm null one in $Y$%
, equivalently, $T$ carries relatively weakly compact subsets of
$X$ to a relatively compact ones in $Y$. The following weak
versions of Dunford-Pettis operators are considered in the Banach
lattice setting. An operator $T:E\rightarrow F$~is

\begin{enumerate}
\item[-] \emph{weak Dunford-Pettis} (wDP), if $f_{n}\left( Tx_{n}\right)
\rightarrow 0$ whenever $(x_{n})$ converges weakly to $0$ in $E$ and $%
(f_{n}) $ converges weakly to $0$ in $F^{\prime }$, equivalently,
$T$ maps relatively weakly compact sets in $E$ to a
Dunford--Pettis sets in $F$ (see Theorem 5.99 \break
of~\cite{AB3}).

\item[-] \emph{almost Dunford-Pettis} (almost DP), if $\left\Vert
Tx_{n}\right\Vert \rightarrow 0$ whenever $\left( x_{n}\right) \subset E$ is
a disjoint weakly null sequence, equivalently, $\left\Vert Tx_{n}\right\Vert
\rightarrow 0$ for every weakly null sequence $\left( x_{n}\right) \subset E$
consisting of positive terms \cite[ Theorem 2.2]{bour}.
\end{enumerate}

The class of wDP operators (between two Banach spaces) was
introduced by C.~D.~Ali\-prantis and O.~Burkinshaw in \cite{AB1}.
It extends the notions of DP operator and the DP property of a
Banach space, in the sense that every DP operator $T:X\rightarrow
Y$ is wDP, and a Banach space $X$ has the DP
property iff the identity operator on $X$ is wDP. Next, J.~A.~Sanchez \cite%
{Sanchez} introduced the class of almost DP operators (from a
Banach lattice into a Banach space) which also extends, in the
Banach lattice setting, the notions of DP operator and the
positive Schur property, since every DP operator $T:E\rightarrow
Y$ is almost DP, and a Banach lattice $E$ has the positive Schur
property iff the identity operator on $E$ is almost DP. Also,
F.~R\"{a}biger \cite{Rabig1} distinguished a class of Banach
lattices with a
weak version of the DP property. A Banach lattice $E$ is said to have the%
\emph{\ weak Dunford-Pettis} \emph{(wDP)} \emph{property} if every weakly
compact operator on $E$ is almost DP, equivalently (see \cite[Proposition 1]%
{expWnuk}), for all sequences $(x_{n})\subset E_{+}^{\prime }$, $%
(f_{n})\subset E^{\prime }$,%
\begin{equation}
x_{n}\overset{\sigma (E,E^{\prime })}{\rightarrow }0,\quad
x_{n}\perp x_{m} \quad \text{and} \quad f_{n}\overset{\sigma
(E,E^{\prime })}{\rightarrow }0 \quad \text{imply} \quad
f_{n}\left( x_{n}\right) \rightarrow 0.  \tag{*}  \label{eq0}
\end{equation}

Inspired by the preceding facts, K. Bouras and M. Moussa introduced
naturally in their recent paper \cite{KM} the class of \emph{almost weak
Dunford-Pettis} (awDP) operators, as a class of operators that extends both
the notions of wDP operator, almost DP operator, and the wDP property of a
Banach lattice. An operator $T:E\rightarrow F$ is said to be awDP, if for
all sequences $(x_{n})\subset E$ , $(f_{n})\subset F^{\prime }$,%
\begin{equation*}
x_{n}\overset{\sigma (E,E^{\prime })}{\rightarrow }0, \quad
x_{n}\perp x_{m} \quad \text{and} \quad f_{n}\overset{\sigma
(E,E^{\prime })}{\rightarrow }0, \quad f_{n}\perp f_{m} \quad
\text{imply} \quad f_{n}\left( Tx_{n}\right) \rightarrow 0.
\end{equation*}

It follows from \cite[Theorem 2.1(3)]{KM} combined with
(\ref{eq0}) that a Banach lattice $E$ has the wDP property iff the
identity operator on $E$ is awDP. Note that the authors gave in
\cite[Theorem 2.1]{KM} some sequence characterisations of positive
almost weak Dunford-Pettis operators, which allowed them to
establish some new sequence characterisations of the wDP property
of a Banach lattice (see \cite[Corollary 2.1]{KM}). The class of
awDP operators contains strictly that of wDP operators as well as
that of almost DP operators, that is, every wDP (resp. almost DP)
operator is awDP, but there exists an awDP operator which is not
wDP nor almost DP. The example of such operator is the identity
operator on a Banach lattice $\Phi $ with the wDP property but
without the DP property nor the positive Schur property. W. Wnuk
gave in \cite[p.~231]{expWnuk} an example of such Banach lattice:

\begin{example}
{\rm{Let $\omega $ be a positive non-increasing continuous
function on $\left(
0,1\right) $ so that%
\begin{equation*}
\lim_{t\rightarrow 0}\omega (t)=\infty \quad \text{and} \quad
\int_{0}^{1}\omega (t)\,dt=1.
\end{equation*}%
The Lorentz function space $E=\wedge (\omega ,1)$ is the space of all
measurable functions $f$ on $\left[ 0,1\right] $ for which%
\begin{equation*}
\left\Vert f\right\Vert _{\omega ,1}=\int_{0}^{1}f^{\ast
}\left( t\right) \omega (t)\,dt<\infty \text{,}
\end{equation*}%
where $f^{\ast }$ denotes the decreasing rearrangement of
$\left\vert f\right\vert $ (cf. \cite[p.~l17,
p.~120]{LindestrausTz}). $E$ is a Banach lattice under the norm
$\left\Vert .\right\Vert _{\omega ,1}$ and the standard almost
everywhere pointwise order (i.e., $f\leq g$ if $f\left( x\right)
\leq g\left( x\right) $ a.e.). Note that, since $E$ has the Fatou
property, then it is a maximal rearrangement invariant space \cite[%
Definition 2.a.1]{LindestrausTz}, and therefore $E$ has a predual,
that is, $E=\Phi ^{\prime }$, where $\Phi $ is the closed linear
span of the simple functions in $E_{i}^{\prime }$. Here,
$E_{i}^{\prime }$ stands for the linear subspace of $E^{\prime }$
consisting of all integrals on $E$, i.e.,
functionals $\varphi _{g}\in E^{\prime }$ defined by%
\begin{equation*}
\varphi _{g}\left( f\right) =\int_{0}^{1}f\left( t\right)
g(t)\,dt ,
\end{equation*}%
where $g$ is any measurable function on $\left[ 0,1\right] $ so
that $gf\in L_{1}\left( 0,1\right) $ for every $f\in E$ (for
details about the preceding facts, see \cite[p.~29, p.~118,
p.~121]{LindestrausTz}). Now, it follows from
\cite[p.~231]{expWnuk} that $\Phi $ has the wDP property but not
the DP property nor the positive Schur property}}.
\end{example}

Moreover, for the class of wDP operators and that of almost DP ones, each
one of the two classes is not included in the other. For instance, as the
Lorentz space $\wedge (\omega ,1)$ has the positive Schur property without
the DP property (see \cite[Remark 3]{Wnuk Schurtype}), the identity operator
on $\wedge (\omega ,1)$ is almost DP but not wDP. Conversely, since $c_{0}$
(the space of real sequences $\left( x_{n}\right) $ with $\lim x_{n}=0$) has
the DP property without the the positive Schur property, the identity
operator on $c_{0}$ is wDP but not almost DP.

The present paper is devoted to the class of awDP operators. W.
Wnuk noted in \cite[Example 4 p.~230]{expWnuk} that a positive
operator $T:E\rightarrow F$\ is almost DP if and only if it is a
DP operator, provided that $F$\ is discrete with order continuous
norm. Motivated by this fact, we look at its weak alternative,
that is, when an awDP operator is wDP? As a response, we prove the
following theorem.

\begin{thA}
\label{thA}Let $E$ and $F$ be two Banach lattices. Then, an order bounded
operator $T:E\rightarrow F$ is almost weak Dunford-Pettis if and only if it
is weak Dunford-Pettis, whenever one of the following holds:

\begin{enumerate}
\item[$\left( i\right) $] $E$ has sequentially weakly continuous lattice
operations.

\item[$\left( ii\right) $] $F^{\prime }$ has sequentially weakly continuous
lattice operations.

\item[$\left( iii\right) $] $T$ is positive and $F$ is discrete with order
continuous norm.
\end{enumerate}
\end{thA}

For that purpose, we present some new characterisations of awDP operators
through almost DP (resp. almost (L)-) sets and some lattice approximations
(Sect. 2). Next, we give the proof of Theorem A and we derive some
consequences (Sect. 3).

We refer the reader to \cite{AB3, MN} for more details on Banach lattice
theory and positive operators.

\section{Characterisation of almost weak Dunford-Pettis operators}

We start this paper by the following lemma which is just a
particular case of Theorem~2.4 of \cite{DF}.

\begin{lemma}
\label{DF}Let $E$ be a Banach lattice, and let $\left( f_{n}\right) \subset
E^{\prime }$ be a sequence with $\left\vert f_{n}\right\vert \overset{%
w^{\ast }}{\rightarrow }0$. If $A\subset E$ is a norm bounded and solid set
such that $f_{n}(x_{n})\rightarrow 0$ for every disjoint sequence $\left(
x_{n}\right) \subset A^{+}:=A\cap E^{+}$, then $\sup_{x\in
A}\left\vert f_{n}\right\vert \left( x\right) \rightarrow 0$.
\end{lemma}

We will use throughout this paper the following lemma.

\begin{lemma}
\label{Lemma fondamental}Let $T:E\rightarrow F$ be an order bounded operator
between two Banach lattices, and let $A$ and $B$ be respectively a norm
bounded solid subsets of $E$ and $F^{\prime }$. Then, the following holds:

\begin{enumerate}
\item If the sequence $\left( f_{n}\right) \subset F^{\prime }$ satisfy $%
\left\vert f_{n}\right\vert \overset{w^{\ast }}{\rightarrow }0$ and $%
f_{n}\left( Tx_{n}\right) \rightarrow 0$ for every disjoint sequence $\left(
x_{n}\right) \subset A^{+}$, then $\left( f_{n}\right) $ converges uniformly
to zero on $T\left( A\right) $.

\item If the sequence $\left( x_{n}\right) \subset E$ satisfy $\left\vert
x_{n}\right\vert \overset{w}{\rightarrow }0$ and $f_{n}\left( Tx_{n}\right)
\rightarrow 0$ for every disjoint sequence $\left( f_{n}\right) \subset
B^{+} $, then $\left( x_{n}\right) $ converges uniformly to zero on $%
T^{\prime }\left( B\right) $.
\end{enumerate}
\end{lemma}
\begin{proof}
$(1)$ We claim that $\left\vert T^{\prime }\left( f_{n}\right)
\right\vert \overset{w^{\ast }}{\rightarrow }0$ holds in
$E^{\prime }$. Let $x\in E^{+}$ and pick some $y\in F^{+}$ such
that $T\left[ -x,x\right] \subseteq \left[ -y,y\right] $. Thus
$$
\left\vert T^{\prime }\left( f_{n}\right) \right\vert
\left( x\right)
=
\sup \left\{ \left\vert T^{\prime }\left( f_{n}\right) \left(
u\right) \right\vert :\left\vert u\right\vert \leq x\right\}
=
\sup \left\{ \left\vert f_{n}\left( T\left( u\right) \right)
\right\vert :\left\vert u\right\vert \leq x\right\} \leq
\left\vert f_{n}\right\vert \left( y\right) \text{.}
$$

Since $\left\vert f_{n}\right\vert \overset{w^{\ast }}{\rightarrow }0$, we
have $\left\vert f_{n}\right\vert \left( y\right) \rightarrow 0$ and hence $%
\left\vert T^{\prime }\left( f_{n}\right) \right\vert \left(
x\right) \rightarrow 0$. Therefore \break
  $\left\vert T^{\prime
}\left( f_{n}\right) \right\vert \overset{w^{\ast }}{\rightarrow
}0$. On the other hand, by hypothesis $T^{\prime }\left(
f_{n}\right) \left( x_{n}\right) =f_{n}\left( Tx_{n}\right)
\rightarrow 0$ for every disjoint sequence $\left(
x_{n}\right) \subset A^{+}$. Then, applying Lemma \ref{DF}, we get $%
\sup_{x\in A}\left\vert T^{\prime }\left( f_{n}\right) \right\vert
\left( x\right) \rightarrow 0$. Now, from the inequality%
\begin{equation*}
\sup_{y\in T\left( A\right) }\left\vert f_{n}\left( y\right)
\right\vert =\sup_{x\in A}\left\vert f_{n}\left( Tx\right)
\right\vert \leq \sup_{x\in A}\left\vert T^{\prime }\left(
f_{n}\right) \right\vert \left( x\right) ,
\end{equation*}%
we conclude that $\sup_{y\in T\left( A\right) }\left\vert f_{n}\left(
y\right) \right\vert \rightarrow 0$ and we are done.

$(2)$ We claim that $\left\vert Tx_{n}\right\vert
\overset{w}{\rightarrow }0$ holds in $F$. Let $f\in \left(
F^{\prime }\right) ^{+}$. By Theorem 1.73 of \cite{AB3},
$T^{\prime }:F^{\prime }\rightarrow E^{\prime }$ is order bounded.
So there exists some $g\in \left( E^{\prime }\right) ^{+}$ such
that $T^{\prime }\left[ -f,f\right] \subseteq \left[ -g,g\right]
$. For each $n$ pick $\left\vert f_{n}\right\vert \leq f$ with
$f\left( \left\vert Tx_{n}\right\vert \right) =f_{n}\left(
Tx_{n}\right) =T^{\prime }\left( f_{n}\right) \left( x_{n}\right)
$ (see Theorem~1.23~\cite{AB3}). Thus, for each $n$, we have
\begin{equation*}
f\left( \left\vert T\left( x_{n}\right) \right\vert \right) =f_{n}\left(
Tx_{n}\right) \leq \left\vert T^{\prime }\left( f_{n}\right) \right\vert
\left( \left\vert x_{n}\right\vert \right) \leq g\left( \left\vert
x_{n}\right\vert \right) \text{.}
\end{equation*}

Since $\left\vert x_{n}\right\vert \overset{w}{\rightarrow }0$, we have $%
g\left( \left\vert x_{n}\right\vert \right) \rightarrow 0$ and hence $%
f\left( \left\vert Tx_{n}\right\vert \right) \rightarrow 0$. Therefore $%
\left\vert Tx_{n}\right\vert \overset{w}{\rightarrow }0$ holds in $F$. On
the other hand, if $j:F\rightarrow F^{\prime \prime }$ is the natural
embedding, then $\left\vert j\left( Tx_{n}\right) \right\vert =j\left(
\left\vert Tx_{n}\right\vert \right) \overset{w^{\ast }}{\rightarrow }0$
holds in $F^{\prime \prime }$. Also, by hypothesis $j\left( Tx_{n}\right)
\left( f_{n}\right) =f_{n}\left( Tx_{n}\right) \rightarrow 0$ for every
disjoint sequence $\left( f_{n}\right) \subset B^{+}$. Then, applying Lemma %
\ref{DF} for the sequence $\left( j\left( Tx_{n}\right) \right) \subset
F^{\prime \prime }$ and the solid subset $B\subset F^{\prime }$, we get $%
\sup_{f\in B}\left\vert j\left( Tx_{n}\right) \right\vert
\left( f\right) \rightarrow 0$. Now, from
\begin{eqnarray*}
\sup_{g\in T^{\prime }\left( B\right) }\left\vert
g\left( x_{n}\right) \right\vert
=
\sup_{f\in B}\left\vert f\left( Tx_{n}\right) \right\vert
\leq \sup_{f\in B}f\left( \left\vert T\left( x_{n}\right)
\right\vert \right)
&= \sup_{f\in B}\left( j\left\vert
Tx_{n}\right\vert \right) \left( f\right)
\\
&= \sup_{f\in
B}\left\vert j\left( Tx_{n}\right) \right\vert \left( f\right)
\text{,}
\end{eqnarray*}%
we conclude that $\sup_{g\in T^{\prime }\left( B\right) }\left\vert
g\left( x_{n}\right) \right\vert \rightarrow 0$. This completes the proof.
\end{proof}

The next result gives a new characterisations of order bounded awDP
operators between Banach lattices, through the almost DP (resp. almost (L)-)
sets.

\begin{theorem}
\label{th carac awDPop}Let $T:E\rightarrow F$ be an order bounded operator
between two Banach lattices. Then the following assertions are equivalent:

\begin{enumerate}
\item $T$ is an awDP operator.

\item $T$ carries the solid hull of each relatively weakly compact subset of
$E$ to an almost DP set in $F$.

\item $T$ carries each relatively weakly compact subset of $E$ to an almost
DP set in $F$.

\item $T^{\prime }$ carries the solid hull of each relatively weakly compact
subset of $F^{\prime }$ to an almost (L)-set in $E^{\prime }$.

\item $T^{\prime }$ carries each relatively weakly compact subset of $%
F^{\prime }$ to an almost (L)-set in $E^{\prime }$.
\end{enumerate}
\end{theorem}

\begin{proof}
$\left( 1\right) \Rightarrow \left( 2\right) $ Let $A$ be a relatively
weakly compact subset of $E$ and let $\left( f_{n}\right) \subset F^{\prime
} $ be a disjoint weakly null sequence. By Theorem 4.34 of \cite{AB3} if $%
\left( x_{n}\right) \subset \left( \mathrm{sol}\left( A\right) \right) ^{+}$
is a disjoint sequence, then $x_{n}\overset{w}{\rightarrow }0$ and hence by
our hypothesis $f_{n}(Tx_{n})\rightarrow 0$. Now, since $\left\vert
f_{n}\right\vert \overset{w}{\rightarrow }0$ by Lemma \ref{Lemma fondamental}
we conclude that $\sup_{x\in T\left( \mathrm{sol}\left( A\right)
\right) }\left\vert f_{n}\left( x\right) \right\vert \rightarrow 0$.
Therefore $T\left( \mathrm{sol}\left( A\right) \right) $ is an almost DP set.

$\left( 2\right) \Rightarrow \left( 3\right) $ Obvious.

$\left( 3\right) \Rightarrow \left( 1\right) $ Let $\left( x_{n}\right)
\subset E$ , $\left( f_{n}\right) \subset F^{\prime }$ be two disjoint
weakly null sequences. It follows that the set $\left\{ Tx_{n}:n\in \mathbb{N%
}\right\} $ is an almost DP. Therefore, $\sup_{k}\left\vert
f_{n}\left( Tx_{k}\right) \right\vert \rightarrow 0$ as $n\rightarrow \infty
$. Now, from the inequality $\sup_{k}\left\vert f_{n}\left(
Tx_{k}\right) \right\vert \geq \left\vert f_{n}\left( Tx_{n}\right)
\right\vert $ we see that $f_{n}\left( Tx_{n}\right) \rightarrow 0$, and
hence $T$ is an awDP operator.

$\left( 1\right) \Rightarrow \left( 4\right) $ Let $B$ be a relatively
weakly compact subset of $F^{\prime }$ and let $\left( x_{n}\right) \subset
E $ be a disjoint weakly null sequence. Similarly, we have $\left\vert
x_{n}\right\vert \overset{w}{\rightarrow }0$ and for each disjoint sequence $%
\left( f_{n}\right) \subset \left( \mathrm{sol}\left( B\right) \right) ^{+}$%
, $f_{n}\overset{w}{\rightarrow }0$ and Thus by hypothesis $%
f_{n}(Tx_{n})\rightarrow 0$. We infer by Lemma \ref{Lemma fondamental} that $%
\sup_{f\in T^{\prime }\left( \mathrm{sol}\left( B\right) \right)
}\left\vert f\left( x_{n}\right) \right\vert \rightarrow 0$ , i.e., $%
T^{\prime }\left( \mathrm{sol}\left( B\right) \right) $ is an almost (L)-set.

$\left( 4\right) \Rightarrow \left( 5\right) $ Obvious.

$\left( 5\right) \Rightarrow \left( 1\right) $ Let $\left( x_{n}\right)
\subset E$ , $\left( f_{n}\right) \subset F^{\prime }$ be two disjoint
weakly null sequences. Since the set $\left\{ T^{\prime }\left( f_{n}\right)
:n\in \mathbb{N}\right\} $ is an almost (L)-set, it follows by the same
justification in $\left( 3\right) \Rightarrow \left( 1\right) $ that $%
f_{n}\left( Tx_{n}\right) \rightarrow 0$, and hence $T$ is an awDP operator.
\end{proof}

The set characterisations in the above theorem enable us to derive
the following result.

\begin{corollary}
\label{th product awDP} In the class of all order bounded operators from $E$
into $E$, the order bounded awDP operators from $E$ into $E$ form a closed
two-sided ideal.
\end{corollary}

\begin{proof}
It is easy to see from the characterisation $\left( 3\right) $ of Theorem %
\ref{th carac awDPop} that if for two operators $T,S:E\rightarrow E$, $S$\
is an awDP operator thus the product $ST$\ is so. Now, if $T$\ is an awDP
operator, let $B$\ be a relatively weakly compact subset of $E^{\prime }$.
Hence, $S^{\prime }\left( B\right) $\ is also a relatively weakly compact
subset of $E^{\prime }$. As $T$\ is an awDP operator then by Theorem \ref{th
carac awDPop}$\left( 5\right) $, $T^{\prime }S^{\prime }\left( B\right) $\
is an almost (L)-set in $E^{\prime }$. This shows that $\left( ST\right)
^{\prime }=T^{\prime }S^{\prime }$\ carries each relatively weakly compact
subset of $E^{\prime }$\ to an almost (L)-set in $E^{\prime }$. Thus, by
Theorem \ref{th carac awDPop}$\left( 5\right) $ again, $ST$\ is an awDP
operator and we are done.
\end{proof}

In our following result, we show that order bounded awDP operators satisfy
some lattice approximations.
\begin{theorem}
\label{th latt app}Let $T:E\rightarrow F$ be an order bounded awDP operator
between two Banach lattices. Then, the following assertions hold:

\begin{enumerate}
\item For each relatively weakly compact subsets $A\subset E$ , $B\subset
F^{\prime }$, and for every $\varepsilon >0$, there exists some $u\in E^{+}$%
satisfying
\begin{equation*}
\left\vert f\right\vert \left( T\left( \left\vert x\right\vert -u\right)
^{+}\right) \leq \varepsilon
\end{equation*}%
for all $x\in A$ and all $f\in B$.

\item For each relatively weakly compact subsets $A\subset E$ , $B\subset
F^{\prime }$, and for every $\varepsilon >0$, there exists some $g\in \left(
F^{\prime }\right) ^{+}$ satisfying
\begin{equation*}
\left( \left\vert f\right\vert -g\right) ^{+}\left( T\left\vert x\right\vert
\right) \leq \varepsilon
\end{equation*}%
for all $x\in A$ and all $f\in B$.
\end{enumerate}
\end{theorem}

\begin{proof}
Note that the proof is similar for the two assertions, so we present only
that of the first one. To do this, we proceed in two steps:

Step 1: For every disjoint sequence $\left( x_{n}\right) $ in the solid hull
of $A$, the sequence $\left( Tx_{n}\right) $ converges uniformly to zero on
the solid hull of $B$. Indeed, the sequence $\left( x_{n}\right) $ is in
this case, a disjoint weakly null sequence (Theorem 4.34 of \cite{AB3}).
Thus, by Theorem \ref{th carac awDPop}$\left( 4\right) $ we have
\begin{equation*}
\sup_{f\in \mathrm{sol}\left( B\right) }\left\vert f\left( Tx_{n}\right)
\right\vert =\sup_{g\in T^{\prime }\left( \mathrm{sol}\left( B\right)
\right) }\left\vert g\left( x_{n}\right) \right\vert \rightarrow 0
\end{equation*}%
as desired.

Step 2: Assume by way of contradiction that there exists a relatively weakly
compact subsets $A\subset E$ , $B\subset F^{\prime }$ and some $\varepsilon
>0$ such that for each $u\in E^{+}$ we have
\begin{equation*}
\left\vert f\right\vert \left( T\left( \left\vert x\right\vert -u\right)
^{+}\right) >\varepsilon
\end{equation*}%
for at least two elements $x\in A$ and $f\in B$. In particular, an easy
inductive argument shows that there exists a sequences $\left( x_{n}\right)
\subset A$ , $\left( f_{n}\right) \subset B$ such that%
\begin{equation}
\left\vert f_{n}\right\vert \left( T\left( \left\vert x_{n+1}\right\vert
-4^{n}\sum_{i=1}^{n}\left\vert x_{i}\right\vert \right) ^{+}\right)
>\varepsilon  \tag{**}  \label{eq 1}
\end{equation}%
holds for each $n$. Put $y=\sum_{n=1}^{\infty }2^{-n}\left\vert
x_{n}\right\vert $ , $y_{n}=\left( \left\vert x_{n+1}\right\vert
-4^{n}\sum_{i=1}^{n}\left\vert x_{i}\right\vert \right)
^{+}$ and   \break    $z_{n}=\left( \left\vert x_{n+1}\right\vert
-4^{n}\sum_{i=1}^{n}\left\vert x_{i}\right\vert -2^{-n}y\right) ^{+}$%
. From Lemma 4.35 of \cite{AB3} the sequence $\left( z_{n}\right) $ is
disjoint. Also, since $0\leq z_{n}\leq |x_{n}+1|$ holds, we see that $\left(
z_{n}\right) \subset sol\left( A\right) $, and so by Step 1 $sup_{f\in
sol\left( B\right) }\left\vert f\left( Tz_{n}\right) \right\vert \rightarrow
0$. In particular, $\left\vert f_{n}\right\vert \left( Tz_{n}\right)
\rightarrow 0$. On the other hand, we have $0\leq y_{n}-z_{n}\leq 2^{-n}y$
from which we get $\left\Vert y_{n}-z_{n}\right\Vert \leq 2^{-n}\left\Vert
y\right\Vert $. In particular, we infer that $\left\vert f_{n}\right\vert
\left( T\left( y_{n}-z_{n}\right) \right) \rightarrow 0$. Therefore, we see
that%
\begin{equation*}
\left\vert f_{n}\right\vert \left( Ty_{n}\right) =\left\vert
f_{n}\right\vert \left( T\left( y_{n}-z_{n}\right) \right) +\left\vert
f_{n}\right\vert \left( Tz_{n}\right) \rightarrow 0\text{,}
\end{equation*}%
which contradicts (\ref{eq 1}). This completes the proof.
\end{proof}

\begin{corollary}
If $T:E\rightarrow F$ is a positive operator between two Banach lattices,
then the following assertions are equivalent:

\begin{enumerate}
\item $T$ is an awDP operator.

\item For each relatively weakly compact subsets $A\subset E$ , $B\subset
F^{\prime }$, and for every $\varepsilon >0$, there exists some $u\in E^{+}$%
satisfying
\begin{equation*}
\left\vert f\right\vert \left( T\left( \left\vert x\right\vert -u\right)
^{+}\right) \leq \varepsilon
\end{equation*}%
for all $x\in A$ and all $f\in B$.

\item For each relatively weakly compact subsets $A\subset E$ , $B\subset
F^{\prime }$, and for every $\varepsilon >0$, there exists some $g\in \left(
F^{\prime }\right) ^{+}$ satisfying
\begin{equation*}
\left( \left\vert f\right\vert -g\right) ^{+}\left( T\left\vert x\right\vert
\right) \leq \varepsilon
\end{equation*}%
for all $x\in A$ and all $f\in B$.
\end{enumerate}
\end{corollary}

\begin{proof}
Note that the proof is similar for the two equivalences $1\Leftrightarrow 2$
and $1\Leftrightarrow 3$. The implication $1\Rightarrow 2$ is exactly
Theorem \ref{th latt app}$(1)$. For the reciprocal one, let $\left(
x_{n}\right) \subset E$ , $\left( f_{n}\right) \subset F^{\prime }$ be two
disjoint weakly null sequences, and let $\varepsilon >0$. Put $A=\left\{
x_{n}:n\in \mathbb{N}\right\} $, $B=\left\{ f_{n}:n\in \mathbb{N}\right\} $.
By hypothesis there exists some $u\in E^{+}$ so that $\left\vert
f\right\vert \left( T\left( \left\vert x\right\vert -u\right) ^{+}\right)
\leq \varepsilon $ holds for all $x\in A$ and all $f\in B$. In particular, $%
\left\vert f_{n}\right\vert \left( T\left( \left\vert x_{n}\right\vert
-u\right) ^{+}\right) \leq \varepsilon $ for all $n$. Now, as $\left\vert
f_{n}\right\vert \overset{w}{\rightarrow }0$ choose some natural number $m$
such that $\left\vert f_{n}\right\vert \left( Tu\right) \leq \varepsilon $
holds for every $n\geq m$. Thus, for every $n\geq m$ we get%
\begin{eqnarray*}
\left\vert f_{n}\left( Tx_{n}\right) \right\vert &\leq &\left\vert
f_{n}\right\vert \left( T\left\vert x_{n}\right\vert \right)
\\
&\leq &\left\vert f_{n}\right\vert \left( T\left( \left\vert
x_{n}\right\vert -u\right) ^{+}\right) +\left\vert f_{n}\right\vert \left(
Tu\right)
\leq 2\varepsilon \text{.}
\end{eqnarray*}

This shows that $f_{n}\left( Tx_{n}\right) \rightarrow 0$, and then $T$ is
an awDP operator.
\end{proof}

\begin{corollary}
For a Banach lattice $E$, the following assertions are equivalent:

\begin{enumerate}
\item $E$ has the wDP property.

\item For each relatively weakly compact subsets $A\subset E$ , $B\subset
F^{\prime }$, and for every $\varepsilon >0$, there exists some $u\in E^{+}$%
satisfying
\begin{equation*}
\left\vert f\right\vert \left( \left( \left\vert x\right\vert -u\right)
^{+}\right) \leq \varepsilon
\end{equation*}%
for all $x\in A$ and all $f\in B$.

\item For each relatively weakly compact subsets $A\subset E$ , $B\subset
F^{\prime }$, and for every $\varepsilon >0$, there exists some $g\in \left(
F^{\prime }\right) ^{+}$ satisfying
\begin{equation*}
\left( \left\vert f\right\vert -g\right) ^{+}\left( \left\vert x\right\vert
\right) \leq \varepsilon
\end{equation*}%
for all $x\in A$ and all $f\in B$.
\end{enumerate}
\end{corollary}

\section{Proof of Theorem A}

\subsection*{$\left( i\right) $ $E$ has sequentially weakly continuous
lattice operations}

Let $\left( x_{n}\right) \subset E$\ , $\left( f_{n}\right) \subset
F^{\prime }$\ be two weakly null sequences. We shall see that $f_{n}\left(
Tx_{n}\right) \rightarrow 0$. To this end, put $A=sol\left\{ x_{n}:n\in
\mathbb{N}\right\} $\ and $B=sol\left\{ f_{n}:n\in \mathbb{N}\right\} $. We
proceed in two steps:

{\it{Step 1}}: We claim that $g_{n}\left( Tx_{n}\right)
\rightarrow 0$\ for every disjoint sequence $\left( g_{n}\right)
\subset B^{+}$. Let $\left(
g_{n}\right) \subset B^{+}$\ be such a sequence. Thus, we have $g_{n}\overset%
{w}{\rightarrow }0$. As $T$\ is an awDP operator, it follows by Theorem \ref%
{th carac awDPop}$\left( 2\right) $ that $\left( g_{n}\right) $\ converges
uniformly to zero on $T\left( A\right) $, that is, $\sup_{x\in
A}\left\vert g_{n}\left( Tx\right) \right\vert \rightarrow 0$. From the
inequality $\left\vert g_{n}\left( Tx_{n}\right) \right\vert \leq
\sup_{x\in A}\left\vert g_{n}\left( Tx\right) \right\vert $, we
conclude that $g_{n}\left( Tx_{n}\right) \rightarrow 0$.

{\it{Step 2}}: Since the lattice operations in $E$\ are
sequentially weakly
continuous, we have $\left\vert x_{n}\right\vert \overset{w}{\rightarrow }0$%
. Thus, taking into account Step 1, we see by Lemma \ref{Lemma fondamental}$%
\left( 2\right) $, that $\left( x_{n}\right) $\ converges
uniformly to zero on $T^{\prime }\left( B\right) $, i.e.,
$\sup_{f\in B}\left\vert f\left( Tx_{n}\right) \right\vert
\rightarrow 0$. From $\left\vert f_{n}\left( Tx_{n}\right)
\right\vert \leq \sup_{f\in B}\left\vert f\left(
Tx_{n}\right) \right\vert $, we conclude that $f_{n}\left(
Tx_{n}\right) \rightarrow 0$. Therefore $T$\ is a wDP operator.
\subsection*{$\left( ii\right) $ $F^{\prime }$ has sequentially weakly
continuous lattice operations}

Note that the two situations $\left( i\right) $ and $\left( ii\right) $ are
symmetric. Let $\left( x_{n}\right) \subset E$\ , $\left( f_{n}\right)
\subset F^{\prime }$\ be two weakly null sequences. We shall see that $%
f_{n}\left( Tx_{n}\right) \rightarrow 0$. To this end, put $A=sol\left\{
x_{n}:n\in \mathbb{N}\right\} $\ and $B=sol\left\{ f_{n}:n\in \mathbb{N}%
\right\} $. We proceed as in $\left( i\right) $:

{\it{Step 1}}: We claim that $f_{n}\left( Ty_{n}\right)
\rightarrow 0$\ for every disjoint sequence $\left( y_{n}\right)
\subset A^{+}$. Let $\left(
y_{n}\right) \subset A^{+}$\ be such a sequence. Thus, we have $y_{n}\overset%
{w}{\rightarrow }0$. As $T$\ is an awDP operator, it follows by Theorem \ref%
{th carac awDPop}$\left( 4\right) $ that $\left( y_{n}\right) $\ converges
uniformly to zero on $T^{\prime }\left( B\right) $, that is, $%
\sup_{g\in B}\left\vert \left( T^{\prime }g\right) \left(
y_{n}\right) \right\vert \rightarrow 0$. From the inequality $\left\vert
f_{n}\left( Ty_{n}\right) \right\vert =\left\vert T^{\prime }\left(
f_{n}\right) \left( y_{n}\right) \right\vert \leq \sup_{g\in
B}\left\vert \left( T^{\prime }g\right) \left( y_{n}\right) \right\vert $,
we conclude that $f_{n}\left( Ty_{n}\right) \rightarrow 0$.

{\it{Step 2}}: Since the lattice operations in $F^{\prime }$\ are
sequentially weakly continuous, we have $\left\vert f_{n}\right\vert \overset%
{w}{\rightarrow }0$. Thus, taking into account Step 1, we see by Lemma \ref%
{Lemma fondamental}$\left( 1\right) $, that $\left( f_{n}\right) $\
converges uniformly to zero on $T\left( A\right) $, i.e., $\sup_{x\in
A}\left\vert f_{n}\left( Tx\right) \right\vert \rightarrow 0$. From $%
\left\vert f_{n}\left( Tx_{n}\right) \right\vert \leq \sup_{x\in
A}\left\vert f_{n}\left( Tx\right) \right\vert $, we conclude that $%
f_{n}\left( Tx_{n}\right) \rightarrow 0$. Therefore $T$\ is a wDP operator.

\subsection*{$\left( iii\right) $ $T$ is positive and $F$ is discrete with
order continuous norm}

Let us recall that an operator $T:E\rightarrow Y$ is said to be order weakly
compact, if $T$ carries each order bounded subset of $E$ to a relatively
weakly compact one in $Y$, equivalently, $T\left( \left[ 0,x\right] \right) $
is relatively weakly compact in $Y$, for every $x\in E^{+}$. We need to
prove the following claim.

\begin{claim}
\label{th prod ST}The product $ST$ is a wDP operator, for every positive
order weakly compact operator $S:F\rightarrow G$ into an arbitrary Banach
lattice.
\end{claim}

\begin{proof}
Let $\left( x_{n}\right) \subset E$, $\left( f_{n}\right) \subset
G^{\prime } $ be a weakly null sequences. We shall see that \break
$f_{n}\left( ST\left(
x_{n}\right) \right) \rightarrow 0$. To this end, let $\varepsilon >0$. As $%
T $ is an awDP operator, then $ST$ is so (Corollary \ref{th product awDP}).
Therefore, by Theorem \ref{th latt app}, pick some $u\in E^{+}$ such that
\begin{equation*}
\left\vert f_{n}\right\vert \left( ST\left( \left\vert x_{n}\right\vert
-u\right) ^{+}\right) <\varepsilon
\end{equation*}%
holds for all $n$. Now, from the inequalities%
\begin{eqnarray*}
\left\vert Tx_{n}\right\vert -Tu &\leq &T\left\vert x_{n}\right\vert
-T\left( \left\vert x_{n}\right\vert \wedge u\right) \\
&\leq &\left\vert T\left\vert x_{n}\right\vert -T\left( \left\vert
x_{n}\right\vert \wedge u\right) \right\vert
=T\left( \left( \left\vert x_{n}\right\vert -u\right) ^{+}\right)
\end{eqnarray*}%
we see that $\left( \left\vert Tx_{n}\right\vert -Tu\right) ^{+}\leq T\left(
\left( \left\vert x_{n}\right\vert -u\right) ^{+}\right) $ holds for all $n$%
. Thus, for every $n$ we have%
\begin{eqnarray*}
\left\vert f_{n}\left( ST\left( x_{n}\right) \right) \right\vert &\leq
&\left\vert f_{n}\right\vert \left( S\left\vert T\left( x_{n}\right)
\right\vert \right) \\
&\leq &\left\vert f_{n}\right\vert \left( S\left( \left( \left\vert T\left(
x_{n}\right) \right\vert -Tu\right) ^{+}\right) \right) +\left\vert
f_{n}\right\vert \left( S\left( \left\vert T\left( x_{n}\right) \right\vert
\wedge Tu\right) \right) \\
&\leq &\left\vert f_{n}\right\vert \left( ST\left( \left\vert
x_{n}\right\vert -u\right) ^{+}\right) +\left\vert f_{n}\right\vert \left(
S\left( \left\vert T\left( x_{n}\right) \right\vert \wedge Tu\right) \right)
\\
&\leq &\varepsilon +\left\vert f_{n}\right\vert \left( S\left( \left\vert
T\left( x_{n}\right) \right\vert \wedge Tu\right) \right) \text{.}
\end{eqnarray*}

Or, as $F$ is discrete with order continuous norm, then it follows from \cite%
[Proposition~2.5.23]{MN} that the lattice operations in $F$ are
sequentially weakly continuous, and then the sequence $\left(
\left\vert T\left(
x_{n}\right) \right\vert \wedge Tu\right) $ is order bounded weakly null in $%
F^{+}$. It follows from \cite[Corollary 3.4.9]{MN} that
$\left\Vert S\left( \left( \left\vert T\left( x_{n}\right)
\right\vert \wedge Tu\right) \right) \right\Vert \rightarrow 0$.
This shows that  $\lim \sup \left\vert f_{n}\left( ST\left(
x_{n}\right) \right) \right\vert \leq \varepsilon $ . As
$\varepsilon >0$ is arbitrary, we infer that $f_{n}\left( ST\left(
x_{n}\right) \right) \rightarrow 0$ as desired.
\end{proof}

Turning to the proof of the theorem, since the norm of $F$ is order
continuous then by \cite[Theorem 4.9]{AB3} each order interval of $F$ is
weakly compact. Thus, the identity operator $I:F\rightarrow F$ is order
weakly compact. Now, it follows from the preceding claim that $T=IT$ is a
wDP operator as desired.

\begin{corollary}
Let $E$ be a Banach lattice such that the lattice operations in $E$ (resp.
in $E^{\prime }$) are sequentially weakly continuous. Then, $E$ has the wDP
property if and only if it has the DP property.
\end{corollary}

In case the range space is a discrete KB-space, the DP operators and their
three weak classes considered in this paper coincide on positive operators.
The details follow.

\begin{corollary}
Let $E$ and $F$ be two Banach lattices such that $F$ is a discrete KB-space.
Then, for a positive operator $T:E\rightarrow F$ the following statements
are equivalent:

\begin{enumerate}
\item $T$ is an awDP operator.

\item $T$ is a wDP operator.

\item $T$ is an almost DP operator.

\item $T$ is a DP operator.
\end{enumerate}
\end{corollary}

\begin{proof}
It suffices to show that a positive wDP operator $T:E\rightarrow F$ is DP.
To this end, let $\left( x_{n}\right) \subset E$ be a weakly null sequence.
Since $F$ is a discrete KB-space then it is a dual (see \cite[Exercise 5.4.E2%
]{MN}), that is, $F=G^{\prime }$ for some Banach lattice $G$. To see that $%
\left\Vert Tx_{n}\right\Vert \rightarrow 0$ it suffices by \cite[Corollary
2.7]{DF} to show that $\left\vert Tx_{n}\right\vert \overset{w^{\ast }}{%
\rightarrow }0$ and $\left( Tx_{n}\right) \left( y_{n}\right) \rightarrow 0$
for every disjoint bounded sequence $\left( y_{n}\right) \subset G^{+}$.
Note that, since the lattice operations in $F$ are sequentially weakly
continuous then we have $\left\vert Tx_{n}\right\vert \overset{w}{%
\rightarrow }0$. Now, if $\left( y_{n}\right) \subset G^{+}$ is a disjoint
bounded sequence, then by Theorem 2.4.14 of \cite{MN}$\ y_{n}\overset{w}{%
\rightarrow }0$ as the norm of $G^{\prime }=F$ is order continuous. By the
lattice embedding $G\hookrightarrow G^{\prime \prime }$ we see that $y_{n}%
\overset{w}{\rightarrow }0$ in $G^{\prime \prime }=F^{\prime }$. Since $T$
is a wDP operator, then $\left( Tx_{n}\right) \left( y_{n}\right)
=y_{n}\left( Tx_{n}\right) \rightarrow 0$ as desired. This completes the
proof.
\end{proof}

In particular, we obtain a result noted by W. Wnuk in \cite[Proposition 6]%
{expWnuk}.

\begin{corollary}
For a discrete KB-space $E$ the following statements are equivalent:

\begin{enumerate}
\item $E$ has the wDP property.

\item $E$ has the DP property.

\item $E$ has the positive Schur property.

\item $E$ has the Schur property.
\end{enumerate}
\end{corollary}




\end{document}